\documentclass[12pt]{article}
\usepackage[paperwidth=230mm, paperheight=310mm]{geometry}
\usepackage{graphicx}
\usepackage{amssymb}
\usepackage{amsmath}
\usepackage{xcolor} % All revisions use the color switch below.
\usepackage[normalem]{ulem} % Blue strikeouts for the proposed deletions.
\usepackage[hidelinks]{hyperref}
\usepackage{enumerate}
\usepackage{amsthm}
\usepackage{amsmath}
\usepackage{float}
\usepackage{mathtools}
 \usepackage{lineno}

\makeatletter
\renewenvironment{proof}[1][\proofname]{%
   \par\pushQED{\qed}\normalfont%
   \topsep6\p@\@plus6\p@\relax
   \trivlist\item[\hskip\labelsep\bfseries#1\@addpunct{.}]%
   \ignorespaces
}{%
   \popQED\endtrivlist\@endpefalse
}
\makeatother

\numberwithin{equation}{section}

\newtheorem{theorem}{Theorem}
\newtheorem{proposition}[theorem]{Proposition}
 \numberwithin{theorem}{section}
 \newtheorem{corollary}[theorem]{Corollary}
\newtheorem{lemma}[theorem]{Lemma}

\newtheorem{remark}[theorem]{Remark}

\makeatletter
\def\keywords{\xdef\@thefnmark{}\@footnotetext}
\makeatother

\renewcommand{\P}{\mathbb{P}}
\newcommand{\E}{\mathbb{E}}
\newcommand{\R}{\mathbb{R}}

\newcommand{\cD}{\mathcal D}
\newcommand{\cH}{\mathcal H}

\newcommand{\cF}{\mathcal F}

\newcommand{\cZ}{\mathcal Z}

\newcommand{\cG}{\mathcal G}

\newcommand{\eps}{\varepsilon}

\newcommand{\revisioncolor}{black}
\newcommand{\revision}[1]{{\color{\revisioncolor}#1}}
\newif\ifshowdeletions
\showdeletionsfalse
\newsavebox{\deletedbox}
\newcommand{\deleted}[1]{%
  \ifshowdeletions
    \begingroup\color{\revisioncolor}%
    \sbox{\deletedbox}{#1}%
    \rlap{\raisebox{.5ex}{\rule{\wd\deletedbox}{.4pt}}}%
    \usebox{\deletedbox}%
    \endgroup
  \fi}
\newcommand{\Extra}[1]{{\color{\revisioncolor}#1}}
\renewcommand{\Extra}[1]{}

\newcommand{\1}{\mathbf 1}
\newcommand{\dd}{\,\mathrm d}

\newcommand{\Zeros}{\operatorname{Zeros}}
\newcommand{\norm}[1]{\lVert #1\rVert}

\newcommand{\Blue}[1]{{\color{black}#1}}

\title{\revision{Exceptional times of compact support for $\alpha$-stable superprocesses}}
\author{Jieliang Hong$^{*}$  and  Leonid Mytnik$^{\dagger}$}
\date{{\small  {\it  
 $^{*}$Department of Mathematics, Southern University of Science and Technology, \\Shenzhen, China\\
 $^{\dagger}$Faculty of Data and Decision Sciences, Technion, Haifa, Israel\\
$^{*}$E-mail:  {\tt hongjl@sustech.edu.cn}  \\
$^{\dagger}$E-mail:  {\tt leonidm@technion.ac.il}  
  }
  }
  }

\hypersetup{
  pdftitle={Exceptional times of compact support for alpha-stable superprocesses},
  pdfauthor={},
  pdfsubject={Exceptional support times and zero sets with continuous immigration},
  pdfkeywords={stable superprocess, exceptional times, compact support,
    extinction, immigration, squared Bessel process, Hausdorff dimension}
}

\begin{document}

\maketitle
\begingroup\color{\revisioncolor}
\begin{abstract}
\revision{We consider symmetric $\alpha$-stable Dawson--Watanabe
superprocesses on $\R^d$ for every $d\geq1$ and $0<\alpha<2$,
started from an arbitrary nonzero finite measure.
We prove that, almost surely, the times at which the spatial
support is compact and nonempty have Hausdorff dimension one
in every nontrivial interval before extinction.
We also prove that, almost surely, the times at which the support
is contained in any prescribed neighborhood of the extinction
point have Hausdorff dimension one in every nontrivial interval
ending at extinction.
These results extend the corresponding conclusions for
$0<\alpha<2/3$ to the full stable range.
%\BlueStrike{The main ingredient is a uniform zero-set estimate for a Feller
%branching diffusion started from zero with a deterministic
%continuous cumulative immigration function: for every
%$\theta\in(0,1)$, sufficiently small total immigration ensures
%that its zero set on a prescribed open interval has Hausdorff
%dimension at least $\theta$, with arbitrarily high probability.}
\Blue{The key new ingredient is a uniform zero-set theorem of
independent interest for Feller branching diffusions started from
zero with deterministic continuous cumulative immigration.
It shows that sufficiently small total immigration ensures that
the zero set on a prescribed open interval has Hausdorff dimension
arbitrarily close to one, with arbitrarily high probability.
The estimate is uniform over the immigration profile and allows
singular continuous immigration measures.}}
\end{abstract}
\endgroup

\smallskip

\noindent\revision{\textbf{Keywords.} Stable superprocess;
exceptional times; compact support; extinction; continuous immigration;
squared Bessel process; Hausdorff dimension.}

\smallskip
\noindent\revision{\textbf{2020 Mathematics Subject Classification.}
Primary 60J68; secondary 60G52, 60J80, 28A80.}

\section{Introduction}

Let $M_F(\R^d)$ be the space of finite Borel measures on $\R^d$,
equipped with the topology of weak convergence.  For
$\mu\in M_F(\R^d)$ and a bounded Borel function $f$, write
\[
  \mu(f)=\int_{\R^d}f(x)\,\mu(\dd x),\qquad
  \mu(1)=\mu(\R^d),
\]
and denote the closed support of $\mu$ by $S(\mu)$.
\revision{Throughout the paper, $\dim$ denotes Hausdorff dimension.}

Fix $d\geq1$ and $\alpha\in(0,2)$.  Let
\[
  \Delta_\alpha=-(-\Delta)^{\alpha/2}
\]
be the generator of the symmetric $\alpha$-stable process on
$\R^d$.  A symmetric stable Dawson--Watanabe superprocess
$X=(X_t,t\geq0)$, starting from $X_0\in M_F(\R^d)$, is the
continuous $M_F(\R^d)$-valued Markov process characterized by the
martingale problem
\begin{equation}\label{e1.1}
  M_t(f)
  :=X_t(f)-X_0(f)-\int_0^tX_s(\Delta_\alpha f)\,\dd s,
  \qquad f\in C_b^2(\R^d),
\end{equation}
where $M(f)$ is a continuous local martingale with quadratic
variation
\begin{equation}\label{e1.2}
  \langle M(f)\rangle_t=\int_0^tX_s(f^2)\,\dd s.
\end{equation}
\revision{For general background on superprocesses, see
\cite{Dyn91,Per02,Li11}.}
We write $\P_{X_0}$ for its law.  

By letting $f\equiv 1$ in \eqref{e1.1} and \eqref{e1.2}, one easily gets that $X(1)=(X_t(1),t\geq0)$ is a Feller
branching diffusion satisfying
\begin{align*}
  X_t(1)=X_0(1)+\int_0^t\sqrt{X_s(1)}\,\dd B_s,
\end{align*}
where $B=(B_t,t\geq0)$ is a one-dimensional standard Brownian motion starting from $0$.
It follows that if we let
\begin{align*}
  \zeta:=\inf\{t\geq0:X_t(1)=0\},
\end{align*}
then
\begin{equation}\label{e1.3}
  \P_{X_0}(\zeta>t)=1-\exp\left\{-\frac{2X_0(1)}{t}\right\},
  \qquad t>0.
\end{equation}
Throughout the paper we assume $X_0\neq0$.

The jump mechanism of the stable spatial motion gives rise to
instantaneous propagation.  A theorem of Perkins \cite{Per90} implies
that, for every fixed $t>0$,
\begin{equation}\label{e1.4}
  \P_{X_0}\bigl(S(X_t)=\R^d\mid X_t\neq0\bigr)=1.
\end{equation}
Extensions to other spatial motions and branching mechanisms may be
found in Evans and Perkins \cite{EP91} and Li and Zhou \cite{LZ08}.
Thus $S(X_t)=\R^d$ for Lebesgue-a.e. $t\in(0,\zeta)$, almost
surely.  This fixed-time statement does not, however, describe the
exceptional geometry of a sample path.

Define
\begin{equation}\label{e1.5}
  H=\{t>0:S(X_t)\text{ is compact and nonempty}\}.
\end{equation}
For $x\in\R^d$ and $r>0$, let
\begin{equation}\label{e1.6}
  D(x,r)=\{y\in\R^d:\norm{y-x}_\infty<r\},
  \qquad D_r=D(0,r),
\end{equation}
and set
\begin{equation}\label{e1.7}
  H_x^r
  =\{t>0:S(X_t)\subseteq\overline{D(x,r)}
        \text{ and }X_t\neq0\}.
\end{equation}

Tribe \cite{Tri92} proved that there is an $\R^d$-valued random
variable $F$, called the extinction point, such that
\begin{equation}\label{e1.8}
  \frac{X_t}{X_t(1)}\longrightarrow\delta_F
  \quad\text{weakly as }t\uparrow\zeta,
  \qquad \P_{X_0}\text{-a.s.}
\end{equation}
\revision{Tribe also obtained support collapse along a sequence of
times when the spatial generator is bounded, or for the
one-dimensional symmetric stable superprocess with
$\alpha\in(0,1/2)$.}

More recently, Hong and Mytnik \cite{HM25} proved that, for
$\alpha\in(0,2/3)$ and every $d\geq1$, the set $H$ is dense in
$[0,\zeta]$, has full Hausdorff dimension, and has dimension one in
every nontrivial subinterval of $[0,\zeta]$.  They also upgraded
Tribe's sequential support-collapse result to a full-dimensional set
of times.  The restriction $\alpha<2/3$ arose from a fourth-moment
estimate for the instantaneous immigration rate at the boundary of a
killed stable process.  The same paper conjectured that both results
remain true for every $\alpha\in(0,2)$.

Our first theorem proves this conjecture for exceptional compact
supports.

\begin{theorem}\label{t1.1}
Let $d\geq1$.  Let $X$ be a symmetric stable superprocess on
$\R^d$ with index $\alpha\in(0,2)$, starting from
$0\neq X_0\in M_F(\R^d)$.  Then 
\[
 \P_{X_0}\Big( \dim\bigl(H\cap(s,t)\bigr)=1  \ \text{ for every }  \ 0\leq s<t\leq\zeta \Big)=1.
\]
In particular, $H$ is dense in $[0,\zeta]$ and
$\dim(H)=1$.
\end{theorem}

The second theorem gives the corresponding support collapse near
extinction and extends both Tribe \cite{Tri92} and
Hong--Mytnik \cite{HM25} to the full stable range.

\begin{theorem}\label{t1.2}
Let $d\geq1$.  Let $X$ be a symmetric stable superprocess on
$\R^d$ with index $\alpha\in(0,2)$, starting from
$0\neq X_0\in M_F(\R^d)$.  Let $F$ be the extinction point in
\eqref{e1.8}.  Then, for every $r>0$,
\[
  \P_{X_0}\Big(    \dim\bigl(H_F^r\cap(t,\zeta)\bigr)=1  \ \text{ for every }  \  0\leq t<\zeta  \Big)=1.
\]
\end{theorem}

\begin{remark}\label{r1.3}
When $d=1$ and $\alpha\in(1,2)$, the superprocess has a density
$u(t,x)$ \revision{at positive times $t>0$},
which is the nonnegative weak solution of the following SPDE (see Konno-Shiga \cite{KS88}):
\begin{align} \label{ea3.1}
\frac{\partial u(t,x)}{\partial t} =\Delta_\alpha u(t,x)+\sqrt{u(t,x)} \dot{W}(t,x), \ \quad u(t,x)\geq 0,
\end{align}
where $\dot W=\dot W(t,x)$ is a space-time white noise.
\revision{For a general finite initial measure, the initial condition
is understood weakly as $u(t,x)\,\dd x\Longrightarrow X_0$
as $t\downarrow0$.}
Theorem \ref{t1.1} therefore shows that the
solution has compact spatial support on a set of times of Hausdorff
dimension one, although its support is $\R$ for Lebesgue-a.e. time
before extinction.
\end{remark}

We outline the proof and indicate the difference from \cite{HM25}.  Suppose first that $X_0$ is supported in
$D_{R/2}$.  We decompose
\begin{equation}\label{e1.9}
  X_t=V_t^R+W_t^R.
\end{equation}
Here $V^R$ is the superprocess whose stable spatial motion is killed
on leaving $D_R$, while $W^R$ collects the descendants of the mass
which leaves $D_R$.  In particular,
\[
  S(V_t^R)\subseteq\overline{D_R}.
\]
Let $A^R=(A_t^R,t\geq0)$ be the cumulative total mass immigrating
into $W^R$.  Conditionally on $V^R$, the total mass
\[
  Z_t=W_t^R(1)
\]
is a Feller branching diffusion with immigration path $A^R$:
\begin{equation}\label{e1.10}
  Z_t=\int_0^t\sqrt{Z_s}\,\dd B_s+A_t^R,\qquad Z_0=0.
\end{equation}
Whenever $Z_t=0$, we have $W_t^R=0$,  hence 
\eqref{e1.9} gives $X_t=V_t^R$ and the support of $X_t$
is contained in $\overline{D_R}$.

The earlier proof in \cite{HM25} wrote $A_t^R=\int_0^t\dot A_s^R\,\dd s$, where we controlled
$\dot A^R$ uniformly, and compared $Z$ with a squared Bessel
process.  The necessary square-integrability is equivalent to
\[
  \E\left[\int_0^t V_s^R(f_R^2)\,\dd s\right]<\infty,
\]
where $f_R$ is defined later in \eqref{e2.4}.  As shown in
\cite[Proposition 1.4]{HM25}, the above holds if and only if
$\alpha<2/3$.  In contrast, our proof does not estimate
$\dot A^R$.  We use only
\begin{equation}\label{e1.11}
  \E_{X_0}[A_t^R]\leq CX_0(1)tR^{-\alpha},
\end{equation}
which follows from a first-moment identity and is valid for every
$\alpha\in(0,2)$.

%\BlueStrike{The following zero-set estimate is the major ingredient of our proof.}
\Blue{Alongside the superprocess results, we establish the following
uniform zero-set theorem for Feller branching diffusions with
deterministic continuous cumulative immigration. The theorem is
formulated independently of the superprocess setting and supplies
the key estimate for Theorems~\ref{t1.1} and~\ref{t1.2}.}

\begin{theorem}\label{t1.4}
Let $0<\theta<1$, $0<\eta<1$, and let
$J\subset(0,1)$ be a nonempty open interval.  There exists
$\kappa=\kappa(\theta,\eta,J)>0$ such that the following holds:
If $a$ is a deterministic continuous increasing function on
$[0,1]$, $a(0)=0$, $a(1)\leq\kappa$, and
\[
  Z_t=\int_0^t\sqrt{Z_s}\,\dd B_s+a(t),\qquad Z_0=0,
\]
then
\[
  \P\bigl(\dim(\Zeros(Z)\cap J)\geq\theta\bigr)\geq1-\eta.
\]
\end{theorem}
Here and below, $\Zeros(f)=\{t:f(t)=0\}$. \Blue{The essential point is the uniformity in $a$: once $\theta$,
$\eta$, and $J$ are fixed, the same threshold $\kappa$ applies to
every continuous increasing $a$ satisfying $a(0)=0$ and
$a(1)\leq\kappa$. No immigration density or common modulus of
continuity is assumed; in particular, the immigration measure may
be singular continuous. Thus the criterion controls the size of
the zero set through total immigration, rather than a pointwise
bound on its rate.}
%\BlueStrike{To our knowledge, this uniform high-probability
%estimate for the local Hausdorff dimension of the zero set is new.}

\revision{The time horizon $1$ in Theorem~\ref{t1.4} is only a
normalization. For any $T>0$, any deterministic continuous increasing
$a$ on $[0,T]$ with $a(0)=0$, and any nonempty open interval
$J\subset(0,T)$, the same conclusion holds provided
$a(T)\leq T\kappa(\theta,\eta,J/T)$, where
$J/T=\{t/T:t\in J\}$. Indeed, by Lemma~\ref{l3.8},
\[
  \widetilde Z_u=\frac{Z_{Tu}}{T},\qquad
  \widetilde a(u)=\frac{a(Tu)}{T},\qquad
  \widetilde B_u=\frac{B_{Tu}}{\sqrt T},\qquad 0\leq u\leq1,
\]
satisfy the same equation on $[0,1]$, and linear rescaling of
time preserves Hausdorff dimension.}

\revision{The process in Theorem~\ref{t1.4} extends the squared
Bessel model to continuous Stieltjes immigration: $Y_t=4Z_t$
satisfies
$\dd Y_t=2\sqrt{Y_t}\,\dd B_t+4\,\dd a(t)$.
If $\dd a(t)=b(t)\,\dd t$, its time-dependent dimension is $4b(t)$.
Faraud and Goutte~\cite{FG14} study squared Bessel processes with
time-dependent dimension and an additional linear drift, assuming
piecewise continuity and local boundedness of the coefficients.
Li~\cite[Example~2.3 and Theorem~3.1]{Li02} develops more general
inhomogeneous immigration through entrance laws and diffuse
immigration measures. Specializing this framework to a one-point
state space, quadratic branching mechanism $q^2/2$, and immigration
measure $\dd a$ gives the transition laws of $Z$, including singular
continuous immigration. Foucart and Uribe Bravo~\cite{FUB14} study
time-homogeneous continuous-state branching processes with
immigration (CBI). For processes started at zero, they describe the
zero sets by means of random cutouts and obtain criteria
for polarity and results on Hausdorff dimension.}

\revision{\emergencystretch=3em % Allow readable line breaks in the marked deletions.
There is a connection between the size of the zero set
and random covering theory. In the excursion representation of
Section~3, an excursion born at $s_i$ with lifetime $\zeta_i$ is
positive on $(s_i,s_i+\zeta_i)$. The relevant uncovered points lie
outside the union of these intervals, giving a
%\BlueStrike{random cutout} 
\Blue{Poisson cutout} interpretation of our zero-set problem.
\Blue{Fitzsimmons, Fristedt and Shepp~\cite[Corollary~3]{FFS85}
obtain a Hausdorff-dimension formula for the uncovered set of a
Poisson cutout with intensity $\dd s\,\mu(\dd\ell)$, which is
homogeneous in the left-endpoint coordinate $s$. In our excursion
representation, the corresponding intensity is
$\dd a(s)\,2\ell^{-2}\,\dd\ell$, where $\dd a$ is an arbitrary
deterministic atomless finite immigration measure, possibly
singular continuous. Hausdorff-dimension results for inhomogeneous
Poisson cutouts under regularity assumptions on the intensity
measure were obtained by Ojala, Suomala and Wu~\cite{OSW17}.
Related Dvoretzky covering results concern complete limsup
coverage by random intervals; see \cite{AM26,FK21,HK25}.
The contribution of Theorem~\ref{t1.4} is a uniform
high-probability estimate: on a prescribed open interval, it
gives any dimension lower bound below one with arbitrarily
high probability, uniformly over all deterministic atomless
finite immigration measures with sufficiently small total mass.
No regularity assumption beyond atomlessness is imposed.
To our knowledge, this uniform high-probability local dimension
estimate under a small-total-mass condition is new.}
}

%\BlueStrike{Now we give a sketch of the proof of Theorem \ref{t1.4}.}
\Blue{We now sketch the proof of Theorem~\ref{t1.4}. The argument
works directly with the cumulative immigration measure and does
not require a regularity estimate for an instantaneous immigration
rate.}
The Laplace transform of $Z$, together with the branching property,
gives the two-point estimate
\begin{equation}\label{e1.12}
  \frac{\P(Z_s\leq\eps,Z_t\leq\eps)}
       {\P(Z_s\leq\eps)\P(Z_t\leq\eps)}
  \leq
  \exp\left\{2\int_0^s\frac{\dd a(u)}{t-u}\right\},
  \qquad 0<s<t.
\end{equation}
For $\delta>0$, we call $t$ good if the left maximal slope
\[
  M_a(t)=\sup_{0<h\leq t}\frac{a(t)-a(t-h)}{h}
\]
is at most $\delta$.
\revision{The one-sided maximal inequality for the Stieltjes measure
$\dd a$, obtained from the so-called Rising Sun Lemma
\cite[Chapter~3, Corollary~3.6]{SS05}, gives}
\[
  |\{t:M_a(t)>\delta\}|\leq\frac{a(1)}{\delta}.
\]
At a good time $t$, integration by parts bounds the right-hand side
of \eqref{e1.12} by a constant times
$|t-s|^{-2\delta}$. See Lemma \ref{l3.3}.  Using this, one obtains the
uniform second-moment and energy estimates.  Paley--Zygmund and
Frostman's criterion yield a fixed positive probability of a
zero-set of dimension at least any $\theta<1-2\delta$. See  Proposition \ref{p3.5}.

To raise this probability to $1-\eta$, we divide $J$ into many
separated intervals.  The Poissonian excursion representation of the
immigration process supplies independent \Blue{processes} after the immigration
measure is restricted to disjoint neighborhoods of these intervals.
Since
\[
  \mathbb N(\text{excursion lifetime}>u)=\frac{2}{u},
\]
smallness of $a(1)$ makes it unlikely that an excursion born outside
one neighborhood survives long enough to reach its central interval.
On this event the local and global processes agree there.  This
localization simultaneously gives independence and transfers the
Frostman estimate back to the original zero set.

Finally, conditionally on the killed process, we may apply Theorem
\ref{t1.4}  to \eqref{e1.10}.  By the
Feller-branching scaling $t\mapsto X_0(1)t$, we use the bound
\eqref{e1.11} to show that the total rescaled immigration \revision{is small with high probability}
whenever $X_0(1)R^{-\alpha}$ is small.   Taking $R$ large proves
Theorem \ref{t1.1}; taking the initial mass small near
extinction and using \eqref{e1.8} proves Theorem
\ref{t1.2}.\\

\bigskip

\noindent  {\bf \large  Funding Declaration}  \\

Jieliang Hong was supported by the National Natural Science Foundation of China (Grant number: 12571150).   Leonid Mytnik was supported  in part by the ISF grants No. 1985/22 and 2917/26. \\

\noindent  {\bf  \large  AI Declaration} \\
 
The results and principal proofs were obtained by the authors. During the preparation of the manuscript, the authors used OpenAI's ChatGPT to assist with language editing and critical reading. The authors reviewed and verified all changes made with this assistance and take full responsibility for the final manuscript.

\section{Killed superprocesses}

\medskip

\noindent\textit{Convention on constants.}
Constants whose values are unimportant may change from line to line.
Their dependence is displayed when it is relevant.\\

\revision{Here and below $B(x,\eps)=\{y\in\R^d:|y-x|<\eps\}$ is the
Euclidean ball, and $|\cdot|$ denotes the Euclidean norm.}
The fractional Laplacian has the integral representation
\begin{equation}\label{e2.1}
  \Delta_\alpha f(x)
  =\lim_{\eps\downarrow0}
    \int_{\R^d\setminus B(x,\eps)}
      \bigl(f(y)-f(x)\bigr)
      \frac{c_{d,\alpha}}{|y-x|^{d+\alpha}}\,\dd y,
\end{equation}
where
\begin{equation}\label{e2.2}
  c_{d,\alpha}
  =\frac{2^\alpha\Gamma((d+\alpha)/2)}
         {\pi^{d/2}|\Gamma(-\alpha/2)|}.
\end{equation}
See, for example, Kwa\'snicki \cite{Kwa17}.

\revision{Fix $R>0$. Let $Y=(Y_t,t\geq0)$ be the symmetric
$\alpha$-stable process with generator $\Delta_\alpha$, write
$Y_t^{(j)}$ for its $j$th coordinate, and set
$\tau_R=\inf\{t\geq0:Y_t\notin D_R\}$.
Its killed semigroup is
$P_t^Rg(x)=\E_x[g(Y_t);t<\tau_R]$.
The generator $\Delta_\alpha^R$ on $C_0(D_R)$ has domain
\[
  \cD(\Delta_\alpha^R)
  =\left\{g\in C_0(D_R):
    \lim_{h\downarrow0}\frac{P_h^Rg-g}{h}
    \text{ exists in }C_0(D_R)\right\}.
\]
Here $C_0(D_R)$ consists of continuous functions on $D_R$
vanishing at its boundary. We extend the pointwise action of
$\Delta_\alpha^R$ to $f\in C_b^2(\R^d)$ by the following
integral formula, for $x\in D_R$:}
\begin{align}
  \Delta_\alpha^R f(x)
  &=
    \lim_{\eps\downarrow0}
    \int_{D_R\setminus B(x,\eps)}
      \bigl(f(y)-f(x)\bigr)
      \frac{c_{d,\alpha}}{|y-x|^{d+\alpha}}\,\dd y \notag\\
  &\quad
    -f(x)\int_{D_R^c}
      \frac{c_{d,\alpha}}{|y-x|^{d+\alpha}}\,\dd y.
     % \qquad \revision{x\in D_R.}
  \label{e2.3}
\end{align}
The above is the Dirichlet fractional Laplacian on $D_R$;
see Fukushima, Oshima and Takeda \cite{FOT94}.  Define its killing
rate by
\begin{equation}\label{e2.4}
  f_R(x)
  =\int_{D_R^c}\frac{c_{d,\alpha}}{|y-x|^{d+\alpha}}\,\dd y,
  \qquad x\in D_R.
\end{equation}
Then
\begin{equation}\label{e2.5}
  (\Delta_\alpha-\Delta_\alpha^R)f(x)
  =\int_{D_R^c}
     f(y)\frac{c_{d,\alpha}}{|y-x|^{d+\alpha}}\,\dd y
\end{equation}
\revision{and, by \eqref{e2.3},
$\Delta_\alpha^R1(x)=-f_R(x)$ for $x\in D_R$.}

Assume for the moment that
\begin{equation}\label{e2.6}
  S(X_0)\subseteq\overline{D_{R/2}}.
\end{equation}
Let $V^R=(V_t^R,t\geq0)$ be the superprocess starting from $X_0$
whose spatial motion is generated by $\Delta_\alpha^R$.  Thus
\begin{equation}\label{e2.7}
  V_t^R(f)
  =X_0(f)+M_t^{V^R}(f)
    +\int_0^tV_s^R(\Delta_\alpha^R f)\,\dd s
\end{equation}
for functions in the domain of $\Delta_\alpha^R$.
\revision{\deleted{By localization,}
The equality \eqref{e2.7} for $f\in C_b^2(\R^d)$, with
$\Delta_\alpha^Rf$ given by \eqref{e2.3}, is supplied by
Hong--Mytnik \cite[Section~2.1, (2.6), p.~476]{HM25}.
Here $M^{V^R}(f)$ is a continuous local martingale with
\[
  \langle M^{V^R}(f)\rangle_t
  =\int_0^tV_s^R(f^2)\,\dd s.
\]
The required drift integrability in our setting follows from
\eqref{e2.5}, the bound
\[
  |\Delta_\alpha^Rf|
  \leq\norm{\Delta_\alpha f}_\infty+\norm f_\infty f_R,
\]
Proposition~\ref{p2.1}, and the first-moment formula for $V^R$.}
In particular,
\[
  S(V_t^R)\subseteq\overline{D_R},\qquad t\geq0.
\]

For bounded Borel $f$, set
\begin{equation}\label{e2.8}
  A_t^R(f)
  =
  \int_0^t\dd s\int_{D_R}V_s^R(\dd x)
  \int_{D_R^c}
    f(y)\frac{c_{d,\alpha}}{|y-x|^{d+\alpha}}\,\dd y.
\end{equation}
The cumulative total immigration is
\begin{equation}\label{e2.9}
  A_t^R=A_t^R(1)=\int_0^tV_s^R(f_R)\,\dd s.
\end{equation}
\begin{proposition}\label{p2.1}
Let $d\geq1$, $\alpha\in(0,2)$, and suppose
\eqref{e2.6} holds.  There is
$C=C(d,\alpha)<\infty$ such that
\begin{equation}\label{e2.10}
  \E_{X_0}[A_t^R]
  =\int_{\R^d}\P_x(\tau_R\leq t)\,X_0(\dd x)
  \leq CX_0(1)(tR^{-\alpha}\wedge1),
  \qquad t\geq0,
\end{equation}
where $\tau_R$ is the first exit time of the symmetric stable
process from $D_R$.  In particular, $A^R$ is almost surely finite
on compact time intervals and has continuous increasing paths.
\end{proposition}

\begin{proof}
\revision{By Hong--Mytnik
\cite[Section~2.1, (2.6), p.~476]{HM25}, applied with $f\equiv1$,
\[
  V_t^R(1)=X_0(1)+M_t^{V^R}(1)-A_t^R,
  \qquad
  \langle M^{V^R}(1)\rangle_t=\int_0^tV_s^R(1)\,\dd s.
\]
The factor $c_{d,\alpha}$ is included in our definition of $f_R$.
By the first-moment formula and the sub-Markov property of the
killed semigroup,
\[
  \E_{X_0}\bigl[\langle M^{V^R}(1)\rangle_t\bigr]
  =\int_0^tX_0(P_s^R1)\,\dd s
  \leq tX_0(1)<\infty.
\]
Consequently $M^{V^R}(1)$ is a square-integrable martingale with
mean zero. Taking expectations gives
\[
  \E_{X_0}[A_t^R]
  =X_0(1)-X_0(P_t^R1)
  =\int_{\R^d}\P_x(\tau_R\leq t)\,X_0(\dd x).
\]
This argument uses no square-integrability estimate for the
instantaneous immigration rate.}
If $x\in\overline{D_{R/2}}$, then $\tau_R\leq t$ implies that, for
some $1\leq j\leq d$,
\[
  \sup_{0\leq s\leq t}
  |Y_s^{(j)}-Y_0^{(j)}|\geq R/2.
\]
Therefore, L\'evy's maximal inequality for symmetric L\'evy
processes and a union bound yield
\[
  \sup_{x\in\overline{D_{R/2}}}\P_x(\tau_R\leq t)
  \leq
  2\sum_{j=1}^d\P_0(|Y_t^{(j)}|\geq R/2).
\]
By stable scaling and
$\P_0(|Y_1^{(j)}|>u)\leq Cu^{-\alpha}$ for $u\geq1$,
\[
  \sup_{x\in\overline{D_{R/2}}}\P_x(\tau_R\leq t)
  \leq C(tR^{-\alpha}\wedge1).
\]
This proves \eqref{e2.10}; see Bertoin
\cite[Chapter~III]{Ber96} for the maximal inequality and the stable
tail estimate.
\revision{Finally, \eqref{e2.9} and the finiteness of
$\E_{X_0}[A_t^R]$ give a common probability-one event on which
$A_n^R<\infty$ for every integer $n\geq1$.
On this event the defining time integral is locally absolutely
continuous, so $A^R$ has continuous increasing paths.}
\end{proof}

\begin{proposition}\label{p2.2}
Under \eqref{e2.6}, $X$ and $V^R$ can be constructed
on an extension of the probability space together with a
measure-valued process $W^R$ such that
\begin{equation}\label{e2.11}
  X_t=V_t^R+W_t^R,\qquad t\geq0.
\end{equation}
Conditionally on $V^R$, $W^R$ is a Dawson--Watanabe superprocess
with measure-valued immigration $A^R$ from \eqref{e2.8}.  In
particular, conditionally on $V^R$, its total mass
$Z_t=W_t^R(1)$ has the law of the unique nonnegative solution of
\begin{equation}\label{e2.12}
  Z_t=\int_0^t\sqrt{Z_s}\,\dd B_s+A_t^R,\qquad Z_0=0.
\end{equation}
\end{proposition}

\begin{proof}
Condition on $V^R$. Given this path, $A^R$ in \eqref{e2.8}
is a locally finite deterministic measure-valued immigration
path \revision{(by Proposition~\ref{p2.1})}.  The
Poisson cluster construction for a Dawson--Watanabe superprocess with
immigration therefore yields a process $W^R$ satisfying, 
%for bounded
%test functions in the domain of $\Delta_\alpha$,
\revision{for $f\in C_b^2(\R^d)$,}
\begin{equation}\label{e2.13}
  W_t^R(f)
  =M_t^{W^R}(f)
   +\int_0^tW_s^R(\Delta_\alpha f)\,\dd s
   +A_t^R(f),
\end{equation}
with
\[
  \langle M^{W^R}(f)\rangle_t
  =\int_0^tW_s^R(f^2)\,\dd s.
\]
\revision{In the common filtered construction of
\cite[(2.10) and the following paragraph, p.~477]{HM25}, the
martingales $M^{V^R}(f)$ and $M^{W^R}(f)$ are orthogonal.}  By
\eqref{e2.5} and \eqref{e2.8},
\[
  V_s^R(\Delta_\alpha^R f)\,\dd s+\dd A_s^R(f)
  =V_s^R(\Delta_\alpha f)\,\dd s.
\]
Hence $V^R+W^R$ solves the Dawson--Watanabe martingale problem
\eqref{e1.1}, with bracket
\[
  \int_0^t(V_s^R+W_s^R)(f^2)\,\dd s.
\]
Uniqueness of the martingale problem therefore allows us to realize
$X$ by $X=V^R+W^R$ on the enlarged space; see Perkins
\cite[Section~II.5]{Per02} and Li \cite[Chapter~9]{Li11} for
superprocesses with immigration.

Taking $f\equiv1$ in \eqref{e2.13} and using
$\Delta_\alpha1=0$, the standard total-mass characterization of a
Dawson--Watanabe superprocess with immigration shows that, conditionally
on $V^R$, $W^R(1)$ is a Feller branching diffusion with immigration path
$A^R$.  Equivalently, it has the law of the unique nonnegative solution
of \eqref{e2.12}; see Li \cite[Chapters~3 and 9]{Li11}.
\end{proof}

\section{Feller branching diffusion with immigration}

Throughout this section, let
$(\Omega,\cF,(\cF_t)_{0\leq t\leq1},\P)$ be a filtered probability
space satisfying the usual conditions and carrying an
$(\cF_t)$-Brownian motion $B$.  Let $a$ be a deterministic continuous
increasing function on $[0,1]$ with $a(0)=0$, and let $Z$ be the
unique nonnegative $(\cF_t)$-adapted continuous solution of
\begin{equation}\label{e3.1}
  Z_t=\int_0^t\sqrt{Z_s}\,\dd B_s+a(t),\qquad Z_0=0.
\end{equation}  Write
\[
  \cZ=\Zeros(Z)\cap[0,1].
\]
\subsection{Two-point estimate}

For $\lambda\geq0$, define
\begin{equation}\label{e3.2}
  u_t(\lambda)=\frac{2\lambda}{2+\lambda t}.
\end{equation}
Then $\partial_tu_t=-u_t^2/2$ and $u_0(\lambda)=\lambda$.

\begin{lemma}\label{l3.1}
For $0<t<t'\leq1$ and $\lambda,\gamma\geq0$,
\begin{equation}\label{e3.3}
  \E[e^{-\lambda Z_t}]
  =
  \exp\left\{
    -\int_0^tu_{t-s}(\lambda)\,\dd a(s)
  \right\},
\end{equation}
and
\begin{align}
  \E[e^{-\lambda Z_t-\gamma Z_{t'}}]
  &=
  \exp\Bigg\{
    -\int_0^t
       u_{t-s}\bigl(\lambda+u_{t'-t}(\gamma)\bigr)\,\dd a(s)
    \notag\\
  &\hspace{5em}
    -\int_t^{t'}u_{t'-s}(\gamma)\,\dd a(s)
  \Bigg\}.
  \label{e3.4}
\end{align}
\end{lemma}

The proof  is given
in Appendix \ref{a1}.  For $\eps>0$ and $0<t\leq 1$, put
\[
  q_\eps(t)=\P(Z_t\leq\eps).
\]
The quantity $q_\eps(t)$ is strictly positive. See the proof of Lemma \ref{l3.1} in Appendix
\ref{a1}.

\begin{proposition}\label{p3.2}
For $0<t<t'\leq1$ and $\eps>0$,
\begin{equation}\label{e3.5}
  \frac{\P(Z_t\leq\eps,Z_{t'}\leq\eps)}
       {q_\eps(t)q_\eps(t')}
  \leq
  \exp\left\{
    2\int_0^t\frac{\dd a(s)}{t'-s}
  \right\}.
\end{equation}
\end{proposition}

\begin{proof}
Fix $0<t<t'\leq1$.  On an extension of the probability space, let
$B^1$ and $B^2$ be independent standard Brownian motions,
independent of $\cF_t$.  For $u\in[t,t']$, let $Z^1$ and $Z^2$
be the unique nonnegative solutions of
\begin{align}
  Z_u^1
  &=Z_t+\int_t^u\sqrt{Z_s^1}\,\dd B_s^1,
  \qquad Z_t^1=Z_t,                                      \label{e3.5c}\\
  Z_u^2
  &=\int_t^u\sqrt{Z_s^2}\,\dd B_s^2+a(u)-a(t),
  \qquad Z_t^2=0.                                        \label{e3.5d}
\end{align}
By the branching property, conditionally on $\cF_t$,
\[
  (Z_u)_{t\leq u\leq t'}
  \stackrel{d}{=}
  (Z_u^1+Z_u^2)_{t\leq u\leq t'}.
\]
Moreover, $Z^2$ is independent of $\cF_t$ and of $Z^1$.
Hence
\begin{equation}\label{e3.5a}
  \P(Z_t\leq\eps,Z_{t'}\leq\eps) \leq \P(Z_t\leq\eps,\ Z_{t'}^2\leq\eps)
  = q_\eps(t)\P(Z_{t'}^2\leq\eps).
\end{equation}
Also,
\begin{align}\label{e3.5b}
  q_\eps(t') \geq  \P(Z_{t'}^1=0,\ Z_{t'}^2\leq\eps)  =\P(Z_{t'}^1=0)\P(Z_{t'}^2\leq\eps).                 
\end{align}
Conditionally on $Z_t=z$, $Z^1$ is a Feller branching diffusion
started from $z$.  Thus
\[
  \P(Z_{t'}^1=0\mid Z_t=z)
  =\exp\left\{-\frac{2z}{t'-t}\right\},
\]
and hence, by \eqref{e3.3} with $\lambda=2/(t'-t)$,
\begin{align*}
  \P(Z_{t'}^1=0) =\E\left[\exp\left\{-\frac{2Z_t}{t'-t}\right\}\right] =\exp\left\{-2\int_0^t\frac{\dd a(s)}{t'-s}\right\}.
\end{align*}
Combining this with \eqref{e3.5a} and \eqref{e3.5b} gives
\eqref{e3.5}.
\end{proof}

\subsection{Maximal estimate}

For $t>0$, define the left maximal slope
\begin{equation}\label{e3.6}
  M_a(t)
  =\sup_{0<h\leq t}\frac{a(t)-a(t-h)}{h}.
\end{equation}
Given $0<\delta<1/4$, let
\begin{equation}\label{e3.7}
  G_\delta(a)=\{t\in(0,1]:M_a(t)\leq\delta\}.
\end{equation}

\begin{lemma}\label{l3.3}
If $0<t<t'\leq1$ and $t'\in G_\delta(a)$, then
\begin{equation}\label{e3.8}
  \exp\left\{
    2\int_0^t\frac{\dd a(s)}{t'-s}
  \right\}
  \leq e^{2\delta}|t'-t|^{-2\delta}.
\end{equation}
\end{lemma}

\begin{proof}
Set $A(r)=a(t')-a(t'-r)$ for $0\leq r\leq t'$. Since we assume $t'\in G_\delta(a)$, by definition we have $A(r)\leq\delta r$ for all $0\leq r\leq t'$.  Stieltjes
integration by parts yields
\begin{align*}
  \int_0^t\frac{\dd a(s)}{t'-s}
  &=
  \int_{t'-t}^{t'}\frac{\dd A(r)}{r} = \left[\frac{A(r)}r\right]_{t'-t}^{t'}
  +\int_{t'-t}^{t'}\frac{A(r)}{r^2}\,\dd r\\
  &\leq
  \delta+\delta\log\frac{t'}{t'-t}
  \leq
  \delta+\delta\log\frac1{t'-t}.
\end{align*}
Then \eqref{e3.8} follows.
\end{proof}

\begin{lemma}\label{l3.4}
For $0<\delta<1/4$,
\begin{equation}\label{e3.9}
  |(0,1]\setminus G_\delta(a)|
  \leq\frac{a(1)}{\delta}.
\end{equation}
\end{lemma}
\revision{The proof is given in Appendix~\ref{a1}.}

\subsection{Frostman estimate}

\begin{proposition}\label{p3.5}
Fix $0<\delta<1/4$ and $0<\theta<1-2\delta$.  Let
$I\subset(0,1)$ be an open interval of length $\ell$, and put
\[
  g_I=|I\cap G_\delta(a)|.
\]
If $g_I>0$, then
\begin{equation}\label{e3.10}
  \P\bigl(\dim(\cZ\cap\overline I)\geq\theta\bigr)
  \geq
  c_{\delta,\theta}
  \frac{g_I^2}{\ell^{2-2\delta}},
\end{equation}
where the constant is independent of $a$ and $I$.
\end{proposition}

\begin{proof}
For any $\eps>0$, set
\begin{equation}\label{e3.11}
  \nu_\eps(\dd t)
  =q_\eps(t)^{-1}
   \1_{\{Z_t\leq\eps\}}
   \1_{I\cap G_\delta(a)}(t)\,\dd t.
\end{equation}
Then $\E[\nu_\eps(I)]=g_I$.  By Proposition \ref{p3.2} and
Lemma \ref{l3.3}, for any distinct $s,t\in I\cap G_\delta(a)$,
\[
  \frac{\P(Z_s\leq\eps,Z_t\leq\eps)}
       {q_\eps(s)q_\eps(t)}
  \leq C_\delta |t-s|^{-2\delta}.
\]
Hence
\begin{equation}\label{e3.12}
  \E[\nu_\eps(I)^2]
  \leq C_\delta
  \iint_{I^2}|t-s|^{-2\delta}\,\dd s\,\dd t
  \leq C_\delta\ell^{2-2\delta}.
\end{equation}
Paley--Zygmund gives
\begin{equation}\label{e3.13}
  \P\left(\nu_\eps(I)\geq\frac{g_I}{2}\right)
  \geq p_I,
  \quad \text{ where } \quad
  p_I:=c_\delta\frac{g_I^2}{\ell^{2-2\delta}}.
\end{equation}
For a finite measure $\nu$, write
\[
  \mathcal E_\theta(\nu)
  =\iint |t-s|^{-\theta}\,\nu(\dd s)\nu(\dd t).
\]
Since $\theta+2\delta<1$,
\begin{equation*} 
  \E[\mathcal E_\theta(\nu_\eps)]
  \leq C_{\delta,\theta}
  \iint_{I^2}|t-s|^{-\theta-2\delta}\,\dd s\,\dd t
  \leq C_{\delta,\theta}\ell^{2-\theta-2\delta}.
\end{equation*}
Set
\[
  K_I=\frac{2C_{\delta,\theta}
  \ell^{2-\theta-2\delta}}{p_I}.
\]
Use Markov's inequality to get that
\begin{equation}\label{e3.14}
  \P\left( \mathcal E_\theta(\nu_\eps)> K_I  \right)
  \leq K_I^{-1}\E[\mathcal E_\theta(\nu_\eps)]
  \leq \frac{p_I}{2}.
\end{equation}
Then \eqref{e3.13} and \eqref{e3.14} imply that
\begin{equation}\label{e3.15}
  \P\left(
    \nu_\eps(I)\geq g_I/2,
    \ \mathcal E_\theta(\nu_\eps)\leq K_I
  \right)
  \geq \frac{p_I}{2}.
\end{equation}

Let $\eps_n\downarrow0$ be deterministic and define
\[
  E_n
  =
  \left\{
    \nu_{\eps_n}(I)\geq g_I/2,
    \ \mathcal E_\theta(\nu_{\eps_n})\leq K_I
  \right\}.
\]
By \eqref{e3.15}, $\P(E_n)\geq p_I/2$ for every $n$.  Set
\[
  E_\infty
  :=\limsup_{n\to\infty}E_n
  =\bigcap_{m=1}^\infty\bigcup_{n\geq m}E_n.
\]
Since the events $\bigcup_{n\geq m}E_n$ decrease as $m\to\infty$,
\[
  \P(E_\infty)
  =\lim_{m\to\infty}
    \P\!\left(\bigcup_{n\geq m}E_n\right)
  \geq
  \lim_{m\to\infty}\sup_{n\geq m}\P(E_n)
  \geq \frac{p_I}{2}.
\]
For every $\omega\in E_\infty$, infinitely many of the events
$E_n$ occur.  We may therefore choose an increasing sequence
$n_k=n_k(\omega)\uparrow\infty$ such that $\omega\in E_{n_k}$ for
all $k$.  Replacing $\eps_n$ by the corresponding subsequence
$\eps_{n_k}$, and relabelling it again as $\eps_n$, we have
\[
  \nu_{\eps_n}(I)\geq g_I/2,
  \qquad
  \mathcal E_\theta(\nu_{\eps_n})\leq K_I
  \quad\text{for all }n.
\]
Since
\[
  \nu_{\eps_n}(I)^2
  \leq \ell^\theta\mathcal E_\theta(\nu_{\eps_n})
  \leq \ell^\theta K_I,
\]
the total masses are uniformly bounded.

 We regard
$\nu_{\eps_n}$ as finite measures on the compact interval
$\overline I$.  By weak compactness, after passing to a further
subsequence,
\[
  \nu_{\eps_n}\Longrightarrow\nu
\]
for some finite measure $\nu$ on $\overline I$.  Since the constant
function $1$ is continuous on $\overline I$,
\[
  \nu(\overline I)
  =\lim_{n\to\infty}\nu_{\eps_n}(\overline I)
  \geq g_I/2.
\]

Weak convergence on $\overline I$ implies
\[
  \nu_{\eps_n}\otimes\nu_{\eps_n}
  \Longrightarrow
  \nu\otimes\nu
\]
on $\overline I^2$.  Define
\[
  K_\theta(s,t)=
  \begin{cases}
    |t-s|^{-\theta},&s\neq t,\\
    +\infty,&s=t.
  \end{cases}
\]
The kernel $K_\theta$ is nonnegative and lower semicontinuous.
Hence the Portmanteau theorem gives
\[
  \mathcal E_\theta(\nu)
  \leq
  \liminf_{n\to\infty}
  \mathcal E_\theta(\nu_{\eps_n})
  \leq K_I.
\]
In particular, $\nu$ has no atoms, since any atom of positive mass
would give infinite contribution from the diagonal.  Hence
\[
  \nu(\partial I)=0,
  \qquad
  \nu(I)=\nu(\overline I)\geq g_I/2.
\]

\revision{Since $Z$ is continuous on $\overline I$ and
$\nu_{\eps_n}$ is carried by $\{t:Z_t\leq\eps_n\}$,
\[
  0\leq \int_{\overline I} Z_t\,\nu(dt)
  =\lim_{n\to\infty}
    \int_{\overline I} Z_t\,\nu_{\eps_n}(dt)
  \leq\lim_{n\to\infty}
    \eps_n\nu_{\eps_n}(\overline I)=0.
\]
Hence $\nu(I\setminus\cZ)=0$.}

%Let $t_0\in I\setminus\cZ$.  By continuity of $Z$, there are a
%relative open interval $U\Subset I$ containing $t_0$ and $c>0$ such
%that $Z\geq c$ on $U$.  Hence $\nu_{\eps_n}(U)=0$ for all sufficiently
%large $n$.  Since $U$ is open in $\overline I$, the Portmanteau
%theorem gives
%\[
%  \nu(U)\leq\liminf_{n\to\infty}\nu_{\eps_n}(U)=0.
%
%\revision{By second countability of $I$, these intervals $U$
%admit a countable subcover of $I\setminus\cZ$. Hence
%$\nu(I\setminus\cZ)=0$.}
Thus the nonzero measure $\nu|_I$ is % supported on 
\revision{carried by} $\cZ\cap I$, and
\[
  \mathcal E_\theta(\nu|_I)
  \leq\mathcal E_\theta(\nu)<\infty.
\]
Frostman's energy criterion \cite[Chapter~8]{Mat95} gives
$\dim(\cZ\cap I)\geq\theta$.
\revision{After changing the constant, we have
$p_I/2=c_{\delta,\theta}g_I^2\ell^{-(2-2\delta)}$;
thus \eqref{e3.10} follows.}
\end{proof}

\subsection{Proof of Theorem \ref{t1.4}}

Let $C_+=C([0,\infty),[0,\infty))$ be equipped with the topology
of uniform convergence on compact time intervals.  For $w\in C_+$,
put
\[
  \zeta(w)=\sup\{t>0:w(t)>0\}.
\]
We use the excursion space
\[
  \mathcal W=
  \left\{w\in C_+:
  \begin{array}{l}
    w(0)=0,\quad 0<\zeta(w)<\infty,\\
    w(t)>0\ \text{for }0<t<\zeta(w),\\
    w(t)=0\ \text{for }t\geq\zeta(w)
  \end{array}
  \right\},
\]
endowed with the Borel $\sigma$-field inherited from $C_+$.  Write
$Y_t(w)=w(t)$ for the coordinate map.  Let $\mathbb N$ be the
canonical excursion measure of the critical Feller branching
diffusion, normalized by
\[
  \mathbb N(1-e^{-\lambda Y_t})=u_t(\lambda),
  \qquad t>0,\ \lambda>0.
\]
Since $\{\zeta>t\}=\{Y_t>0\}$ on $\mathcal W$, monotone
convergence gives
\begin{equation}\label{e3.16}
  \mathbb N(\zeta>t)
  =\lim_{\lambda\to\infty}
    \mathbb N(1-e^{-\lambda Y_t})
  =\lim_{\lambda\to\infty}u_t(\lambda)
  =\frac2t.
\end{equation}
If
\[
  \Pi=\sum_i\delta_{(s_i,Y^i)}
\]
is a Poisson random measure on $[0,1]\times\mathcal W$ with
intensity $\dd a(s)\mathbb N(\dd Y)$, then
\begin{equation}\label{e3.17}
  Z_t=\sum_{s_i\leq t}Y^i_{t-s_i}
\end{equation}
 has a continuous version with the same law on $C([0,1])$
as the solution of \eqref{e3.1}. Indeed, the Poisson exponential formula and the branching
property give the same transition Laplace transforms as those
of \eqref{e3.1}.
We use continuous versions of all Poisson sums below.
For a general construction of inhomogeneous immigration processes
by Poisson sums, see Li~\cite[Section~4]{Li02}.

 Theorem~\ref{t1.4} follows immediately from the following result,
in which the threshold depends on $J$ only through its length.

%has the law of \eqref{e3.1}.
%\revision{For general inhomogeneous immigration with diffuse
%immigration-time measures and its Poisson--Kuznetsov construction,
%see Li~\cite[Example~2.3, Theorem~3.1 and Lemma~4.1]{Li02};
%see also Li~\cite[Chapter~9]{Li11} and Li and Wang~\cite{LW99}
%for related immigration constructions.
%All Poisson sums below are understood in their continuous versions,
%constructed measurably from the respective restrictions of $\Pi$;
%independence is thereby preserved. Equality at rational times then
%extends to the entire interval by continuity.}

\begin{theorem}\label{t3.6}
Fix $0<\theta<1$, $0<\eta<1$, and an open interval
$J\subset(0,1)$ of length $L>0$.  There is
$\kappa=\kappa(\theta,\eta,L)>0$ such that, for every continuous
increasing $a$ with $a(0)=0$ and $a(1)\leq\kappa$,
\begin{equation}\label{e3.18}
  \P\bigl(\dim(\cZ\cap J)\geq\theta\bigr)\geq1-\eta.
\end{equation}
\end{theorem}

\begin{proof}
Choose $\delta\in(0,1/4)$ so that $\theta<1-2\delta$, and assume
\begin{equation}\label{e3.19}
  \kappa\leq\frac{\delta L}{2}.
\end{equation}
By Lemma \ref{l3.4}, we have 
\begin{equation*} 
  |(0,1]\setminus G_\delta(a)|
  \leq\frac{a(1)}{\delta} \leq \frac{\kappa}{\delta} \leq \frac{  L}{2}.
\end{equation*}
Hence it is immediate that (recall $|J|=L$)
\[
  |J\cap G_\delta(a)|\geq L/2.
\]

Partition $J$ into $3n$ consecutive intervals of length
$\ell=L/(3n)$.  One residue class modulo $3$ contains intervals
$(I_k)_{k\in\mathcal I}$ satisfying
\begin{equation}\label{e3.20}
  \sum_{k\in\mathcal I}g_k\geq L/6,
  \qquad
  g_k:=|I_k\cap G_\delta(a)|>0,
\end{equation}
and distinct $I_k$'s are separated by distance at least $2\ell$.
Let
\[
  E_k=\{s\in[0,1]:\operatorname{dist}(s,I_k)<\ell/2\}.
\]
Then the $E_k$'s are pairwise disjoint.

Let $Z^{(k)}$ be obtained from \eqref{e3.17} by retaining only
Poisson points with $s_i\in E_k$, and define
\[
  a_k(t)=\int_{[0,t]}\1_{E_k}(s)\,\dd a(s).
\]

\revision{Since $a$ is continuous, so is $a_k$.}  Moreover, for
$t\in I_k$ and $0<h\leq t$,
\[
  a_k(t)-a_k(t-h)
  =\int_{(t-h,t]}\1_{E_k}(s)\,\dd a(s)
  \leq a(t)-a(t-h).
\]
Consequently,
\[
  M_{a_k}(t)\leq M_a(t),\qquad t\in I_k,
\]
and therefore
\[
  I_k\cap G_\delta(a)
  \subseteq I_k\cap G_\delta(a_k).
\]
By letting
\[
  g_k'=|I_k\cap G_\delta(a_k)|,
\]
we have $g_k'\geq g_k$.  By the restriction property of Poisson random
measures, $Z^{(k)}$ has the representation \eqref{e3.17} with
$\dd a$ replaced by $\dd a_k=\1_{E_k}\dd a$; hence it has the
law of \eqref{e3.1} with immigration path $a_k$.  Proposition
\ref{p3.5} is therefore applicable to $Z^{(k)}$ on $I_k$.  With
\[
  S_k=
  \{\dim(\Zeros(Z^{(k)})\cap I_k)\geq\theta\},
\]
and using $g_k'\geq g_k$, we obtain
\begin{equation}\label{e3.21}
  \P(S_k)
  \geq c_{\delta,\theta}
  \frac{(g_k')^2}{\ell^{2-2\delta}}
  \geq c_{\delta,\theta}
  \frac{g_k^2}{\ell^{2-2\delta}}.
\end{equation}
Here Proposition \ref{p3.5} is stated with $\overline I_k$;
since $\theta>0$, adding or removing the two endpoints does not
change the Hausdorff-dimension event.
The events $(S_k)_{k\in\mathcal I}$ are independent.  Since
$|\mathcal I|\leq n$, Cauchy--Schwarz and \eqref{e3.20} give
\[
 \frac{L^2}{36} \leq \Big(\sum_{k\in\mathcal I}g_k\Big)^2
  \leq n\sum_{k\in\mathcal I}g_k^2.
\]
Consequently,
\begin{equation}\label{e3.22}
  \sum_{k\in\mathcal I}\P(S_k)
  \geq  \frac{c_{\delta,\theta}}{\ell^{2-2\delta}} \sum_{k\in\mathcal I}   g_k^2
  \geq  \frac{c_{\delta,\theta}}{\ell^{2-2\delta}}  \frac{L^2}{36n} =c_{\delta,\theta,L}\ell^{2\delta-1},
\end{equation}
where in the last equality we use $L/(3n)=\ell$.
Choose $n$  large (i.e. $\ell$ small) so that
\begin{equation}\label{e3.23}
  \P\left(\bigcap_{k\in\mathcal I}S_k^c\right)=\prod _{k\in\mathcal I} (1-\P(S_k))
  \leq
  \exp\left\{-\sum_{k\in\mathcal I}\P(S_k)\right\}
  \leq\eta/2.
\end{equation}

Set
\[
  N_\ell
  =\#\{i:s_i\in[0,1],\ \zeta(Y^i)>\ell/4\}.
\]
By \eqref{e3.16}, $N_\ell$ is Poisson with mean
\[
  \frac{8a(1)}{\ell}.
\]
\revision{After fixing $n$, choose
\[
  \kappa\leq\min\{\delta L/2,\eta\ell/16\}.
\]
Then}
\begin{equation}\label{e3.24}
  \P(N_\ell\geq1) \leq \E[N_\ell]=\frac{8a(1)}{\ell}
  \leq\frac{8\kappa}{\ell}
  \leq\eta/2.
\end{equation}
If $N_\ell=0$, then for $t\in I_k$ every excursion born at
$s\notin E_k$ and $s<t$ is zero at age $t-s$ since
$t-s\geq\ell/2$.  Hence
\[
  Z_t=Z_t^{(k)},\qquad t\in I_k.
\]
Thus, on $\{N_\ell=0\}\cap\bigcup_{k\in\mathcal I}S_k$,
$\dim(\cZ\cap J)\geq\theta$.  Equations \eqref{e3.23} and
\eqref{e3.24} prove \eqref{e3.18}.
\end{proof}

The following conditional form will be used below.

\begin{corollary}\label{c3.7}
Let $A$ be a random continuous increasing function on $[0,1]$
with $A_0=0$.
\revision{Let $\cG$ be a sub-$\sigma$-field such that $A$ is
$\cG$-measurable as a $C([0,1])$-valued random variable.
Suppose that, for almost every $\omega$, the regular conditional
law of $Z$ given $\cG$ is the law of the solution of
\eqref{e3.1} with deterministic immigration $a=A(\omega)$.}
For the parameters in Theorem \ref{t3.6},
\begin{equation}\label{e3.25}
  \P\bigl(\dim(\Zeros(Z)\cap J)\geq\theta\bigr)
  \geq
  1-\eta-\P(A_1>\kappa).
\end{equation}
\end{corollary}

\begin{proof}
\revision{Condition on $\cG$} and apply Theorem \ref{t3.6} on
$\{A_1\leq\kappa\}$.
\end{proof}

\begin{lemma}\label{l3.8}
\revision{Let $T>0$ and suppose that $Z$ satisfies}
\[
  Z_t=\int_0^t\sqrt{Z_s}\,\dd B_s+a(t),\qquad Z_0=0,
\]
\revision{for $0\leq t\leq T$.} Then, for every $m>0$,
\[
  \widetilde Z_u=\frac{Z_{mu}}m,\qquad
  \widetilde a(u)=\frac{a(mu)}m,
  \revision{\qquad 0\leq u\leq T/m,}
\]
satisfy
\[
  \widetilde Z_u
  =
  \int_0^u\sqrt{\widetilde Z_v}\,\dd\widetilde B_v
  +\widetilde a(u),
\]
where $\widetilde B_u=m^{-1/2}B_{mu}$ is a standard Brownian
motion \revision{on $[0,T/m]$}.
\end{lemma}

\begin{proof}
Since $B_{mv}=\sqrt m\,\widetilde B_v$,
\[
  \frac1m\int_0^{mu}\sqrt{Z_s}\,\dd B_s
  =\int_0^u\sqrt{\frac{Z_{mv}}m}\,\dd\widetilde B_v.
\]
Dividing the equation for $Z_{mu}$ by $m$ gives the result.
\end{proof}

\section{Local support estimate}

\begin{proposition}\label{p4.1}
Fix $d\geq1$, $\alpha\in(0,2)$,
$0<\theta<1$, and $0<\eta<1$.  There exists
$\rho_0\in(0,1/4)$ such that, for every
$0<\rho\leq\rho_0$, there is
$\gamma=\gamma(d,\alpha,\theta,\eta,\rho)>0$ with the following
property.

Let $0\neq\mu\in M_F(\R^d)$, put $m=\mu(1)$, and suppose that
for some $y\in\R^d$ and $R>0$,
\begin{equation}\label{e4.1}
  \mu(D(y,R/2)^c)\leq\rho^3m,
  \qquad
  mR^{-\alpha}\leq\gamma.
\end{equation}
Then
\begin{equation}\label{e4.2}
  \P_\mu\left(
    \dim\bigl(H_y^R\cap(\rho^2m,\rho m)\bigr)\geq\theta
  \right)
  \geq1-\eta.
\end{equation}
\end{proposition}

\begin{proof}
By translation invariance, take $y=0$.  Split
\[
  \mu^1=\mu|_{D_{R/2}},
  \qquad
  \mu^2=\mu|_{D_{R/2}^c},
\]
and realize $X=X^1+X^2$ as the sum of independent stable
superprocesses with initial states $\mu^1$ and $\mu^2$.
Let $\zeta_i$ be the extinction time of $X^i$.
Since $\mu^1(1)\geq(1-\rho^3)m$, formula
\eqref{e1.3} gives, for $\rho<1/4$,
\begin{align}
  \P(\zeta_1\leq\rho m)
  &=
  \exp\left\{-\frac{2\mu^1(1)}{\rho m}\right\}
  \leq e^{-1/\rho},
  \label{e4.3}\\
  \P(\zeta_2>\rho^2m)
  &\leq\frac{2\mu^2(1)}{\rho^2m}
  \leq2\rho.
  \label{e4.4}
\end{align}
Choose $\rho_0$ so small that
\begin{equation}\label{e4.5}
  e^{-1/\rho}+2\rho\leq\eta/3,
  \qquad 0<\rho\leq\rho_0.
\end{equation}

Apply Proposition \ref{p2.2} to $X^1$:
\[
  X_t^1=V_t^R+W_t^R.
\]

Define
\[
  \widetilde Z_u=\frac{W_{mu}^R(1)}m,
  \qquad
  a(u)=\frac{A_{mu}^R}{m},
  \qquad 0\leq u\leq1.
\]
By Proposition~\ref{p2.2} and Lemma~\ref{l3.8},
conditionally on $V^R$, $\widetilde Z$ is a Feller branching
diffusion with immigration path $a$.

Let $J=(\rho^2,\rho)$, and let $\kappa=\kappa(\theta,\eta/3,J)$ be the
constant in Theorem \ref{t3.6}.  Conditionally on $V^R$, the path $a$ is
deterministic.  Hence, on the event $\{a(1)\leq\kappa\}$, Theorem
\ref{t3.6} gives
\[
  \P\left(
    \left.
    \dim\bigl(\Zeros(\widetilde Z)\cap J\bigr)\geq\theta
    \ \right|\ V^R
  \right)
  \geq1-\frac{\eta}{3}.
\]
Equivalently, Corollary \ref{c3.7},
\revision{with $\cG=\sigma(V_s^R:s\geq0)$,} yields
\begin{equation}\label{e4.6a}
  \P\left(
    \dim\bigl(\Zeros(\widetilde Z)\cap J\bigr)\geq\theta
  \right)
  \geq1-\frac{\eta}{3}-\P(a(1)>\kappa).
\end{equation}

Proposition \ref{p2.1} gives
\begin{equation}\label{e4.6}
  \E[a(1)]
  =
  \frac{\E[A_m^R]}m
  \leq C\mu^1(1)R^{-\alpha}
  \leq CmR^{-\alpha}.
\end{equation}
Choose 
\begin{equation*} 
0<\gamma\leq\frac{\eta\kappa}{3C}
\end{equation*}
such that \eqref{e4.6} implies
\begin{equation*} 
 \E[a(1)] \leq CmR^{-\alpha} \leq \frac{\eta\kappa}{3}.
\end{equation*}
Then Markov's inequality implies
\[
  \P(a(1)>\kappa)
  \leq \frac{\E[a(1)]}{\kappa}
  \leq\frac{\eta}{3}.
\]
Substituting this into \eqref{e4.6a},
\[
  \P\left(
    \dim\bigl(
      \Zeros(\widetilde Z)\cap(\rho^2,\rho)
    \bigr)\geq\theta
  \right)
  \geq1-\frac{2\eta}{3}.
\]
For $t=mu$,
\[
  W_t^R(1)=0
  \quad\Longleftrightarrow\quad
  \widetilde Z_{t/m}=0.
\]
Thus, writing $mA=\{mu:u\in A\}$,
\[
  \Zeros(W^R(1))\cap(\rho^2m,\rho m)
  =m\bigl(\Zeros(\widetilde Z)\cap(\rho^2,\rho)\bigr).
\]
The map $u\mapsto mu$ is bi-Lipschitz, so Hausdorff dimension is
unchanged.  Hence
\begin{equation}\label{e4.8}
  \P\left(
    \dim\bigl(
      \Zeros(W^R(1))\cap(\rho^2m,\rho m)
    \bigr)\geq\theta
  \right)
  \geq1-\frac{2\eta}{3}.
\end{equation}

On the event
\[
  \{\zeta_2\leq\rho^2m<\rho m<\zeta_1\},
\]
the process $X^2$ has vanished throughout
$(\rho^2m,\rho m)$, while $X^1$ is nonzero there.  At every
zero of $W^R(1)$ in this interval,
\[
  X_t=X_t^1=V_t^R,
\]
and hence $X_t\neq0$ and
$S(X_t)\subseteq\overline{D_R}$.  Combining
\eqref{e4.5} and \eqref{e4.8} proves
\eqref{e4.2}.
\end{proof}

\subsection{Proofs of Theorems \ref{t1.1} and \ref{t1.2}}

The globalization arguments are the same as those in
\cite[Sections~2.2--2.3]{HM25}; we record only the points where
Proposition~\ref{p4.1} is used.

\begin{proof}[Proof of Theorem \ref{t1.1}]
The argument is the same as in \cite[Section~2.2]{HM25}, with
Proposition~\ref{p4.1} replacing the corresponding local support
estimate there.  We only indicate the reduction.  Fix $T>0$,
$0<\theta<1$, and $0<\eta<1$, and put $m=X_0(1)$.  Choose
$0<\rho\leq\rho_0$ so that $\rho m<T$, and then choose $R$ large
such that
\[
  X_0(D_{R/2}^c)\leq\rho^3m,
  \qquad
  mR^{-\alpha}\leq\gamma,
\]
where $\gamma$ is given by Proposition~\ref{p4.1}.  Hence
\[
  \P_{X_0}\bigl(\dim(H\cap(0,T))\geq\theta\bigr)\geq1-\eta.
\]
Letting $\eta\downarrow0$ and $\theta\uparrow1$ gives
$\dim(H\cap(0,T))=1$ almost surely.  Applying the same statement
after every rational time by the strong Markov property, exactly as
in \cite[Section~2.2]{HM25}, yields
\[
  \dim(H\cap(s,t))=1,
  \qquad 0\leq s<t\leq\zeta,
\]
almost surely.  The remaining assertions follow immediately.
\end{proof}

\begin{proof}[Proof of Theorem \ref{t1.2}]
We follow \cite[Section~2.3]{HM25},
\revision{using Proposition~\ref{p4.1} at stopping times determined
by the total mass.}
Fix $r>0$, choose
$\{x_k:k\geq1\}\subset\R^d$ with
\[
  \R^d=\bigcup_{k\geq1}D(x_k,r/4),
\]
and continuous $\phi_k^r:\R^d\to[0,1]$ such that
$\phi_k^r=0$ on $D(x_k,r/4)$ and $\phi_k^r=1$ on
$D(x_k,r/2)^c$.  Put
\[
  \cH^r=\bigcup_{k\geq1}H_{x_k}^r.
\]
Fix $0<\theta<1$ and $0<\eta<1$, choose
$0<\rho\leq\rho_0$, and let $\gamma$ be given by
Proposition~\ref{p4.1}.

% BEGIN BLUE REVISION: stopping at a prescribed total mass.
\begingroup
\color{\revisioncolor}Let $(\cF_t)_{t\geq0}$ be the completed,
right-continuous natural filtration of $X$.
For $0<\eps<X_0(1)$, define the stopping time
\[
  \tau_\eps=\inf\{t\geq0:X_t(1)\leq\eps\}.
\]
By continuity of the total mass,
\[
  0<\tau_\eps<\zeta,\qquad
  X_{\tau_\eps}(1)=\eps,\qquad
  \tau_\eps\uparrow\zeta
  \quad\text{as }\eps\downarrow0,
  \qquad \P_{X_0}\text{-a.s.}
\]
Let
\[
  B_\eps=
  \left\{
    X_{\tau_\eps}(\phi_k^r)\leq\rho^3\eps
    \text{ for some }k\geq1
  \right\}.
\]
Then $B_\eps\in\cF_{\tau_\eps}$.  Almost surely, there is an
index $k\geq1$ such that $F\in D(x_k,r/4)$.  For this index,
\eqref{e1.8} gives
\[
  \frac{X_{\tau_\eps}(\phi_k^r)}{\eps}
  =\frac{X_{\tau_\eps}(\phi_k^r)}{X_{\tau_\eps}(1)}
  \longrightarrow\phi_k^r(F)=0
  \qquad\text{as }\eps\downarrow0.
\]
Since $\rho$ is fixed, $B_\eps$ holds for all sufficiently small
$\eps$, almost surely.  Consequently,
\[
  \P_{X_0}(B_\eps)\longrightarrow1.
\]

On $B_\eps$, set
\[
  k_\eps=\min\left\{k\geq1:
    X_{\tau_\eps}(\phi_k^r)\leq\rho^3\eps
  \right\},
\]
and set $k_\eps=1$ on $B_\eps^c$.  This choice is
$\cF_{\tau_\eps}$-measurable, and on $B_\eps$ we have
\[
  X_{\tau_\eps}\bigl(D(x_{k_\eps},r/2)^c\bigr)
  \leq X_{\tau_\eps}(\phi_{k_\eps}^r)
  \leq\rho^3\eps.
\]
Write
\[
  G_\eps=
  \left\{\dim\bigl(\cH^r\cap(\tau_\eps,\infty)\bigr)
         \geq\theta\right\}.
\]
For $\eps$ small enough that $\eps r^{-\alpha}\leq\gamma$,
the strong Markov property and Proposition~\ref{p4.1}, applied
with initial measure $X_{\tau_\eps}$, center $x_{k_\eps}$, and
radius $r$, give
\[
  \P_{X_0}\!\left(G_\eps\,\middle|\,\cF_{\tau_\eps}\right)
  \geq1-\eta
  \qquad\text{on }B_\eps.
\]
Indeed, the proposition gives the dimension bound for
$H_{x_{k_\eps}}^r\subseteq\cH^r$ in the interval
$(\tau_\eps+\rho^2\eps,\tau_\eps+\rho\eps)$.
It follows that
\[
  \P_{X_0}(B_\eps\cap G_\eps)
  \geq(1-\eta)\P_{X_0}(B_\eps).
\]

For every $N\geq1$, since $\cH^r\subset(0,\zeta)$,
\begin{align*}
  &\P_{X_0}\!\left(
    \dim\bigl(\cH^r\cap(\zeta-N^{-1},\zeta)\bigr)\geq\theta
  \right)\\
  &\qquad\geq
  \P_{X_0}(B_\eps\cap G_\eps)
  -\P_{X_0}(\tau_\eps\leq\zeta-N^{-1})\\
  &\qquad\geq
  (1-\eta)\P_{X_0}(B_\eps)
  -\P_{X_0}(\tau_\eps\leq\zeta-N^{-1}).
\end{align*}
Letting $\eps\downarrow0$ gives a lower bound of $1-\eta$.
The event on the left does not depend on $\eta$.  Letting
$\eta\downarrow0$, then taking a countable sequence
$\theta\uparrow1$ and intersecting over $N\geq1$, gives
\[
  \dim\bigl(\cH^r\cap(\zeta-N^{-1},\zeta)\bigr)=1
  \quad\text{for every }N\geq1,\qquad \P_{X_0}\text{-a.s.}
\]
\par
\endgroup
% END BLUE REVISION.

Finally, choose $k_0$ with $F\in D(x_{k_0},r/4)$.  By
\eqref{e1.8}, $X_t(D(x_{k_0},r/2))>0$ for all $t$ sufficiently
close to $\zeta$.  If such a $t$ lies in $\cH^r$, then
$S(X_t)\subseteq\overline{D(x_k,r)}$ for some $k$, so
$\overline{D(x_k,r)}$ intersects $D(x_{k_0},r/2)$.  Consequently
$\|x_k-x_{k_0}\|_\infty<3r/2$, and hence
\[
  S(X_t)\subseteq\overline{D(F,3r)}.
\]
Thus $\cH^r\subseteq H_F^{3r}$ sufficiently close to extinction,
and the preceding full-dimensional statement yields
\[
  \dim\bigl(H_F^{3r}\cap(t,\zeta)\bigr)=1,
  \qquad t<\zeta.
\]
Replacing $r$ by $r/3$ proves the theorem.
\end{proof}

\appendix

\section{Auxiliary lemmas}\label{a1}

\subsection{Lemma \ref{l3.1}}

\begin{proof}[Proof of Lemma \ref{l3.1}]
Fix $t>0$ and $\lambda\geq0$. Recall that
\begin{equation} 
  u_t(\lambda)=\frac{2\lambda}{2+\lambda t},
\end{equation}
and $\partial_t u_t=-u_t^2/2$. Set
\[
  v(s)=u_{t-s}(\lambda),\qquad 0\leq s\leq t.
\]
Then $v'(s)=v(s)^2/2$.  Define
\begin{equation}\label{e6.1}
  N_s
  =
  \exp\left\{
    -v(s)Z_s+\int_0^s v(r)\,\dd a(r)
  \right\},\qquad 0\leq s\leq t.
\end{equation}
Set
\[
  H_s=-v(s)Z_s+\int_0^s v(r)\,\dd a(r),
  \qquad N_s=e^{H_s}.
\]
Since
\[
  \dd Z_s=\sqrt{Z_s}\,\dd B_s+\dd a(s),
\]
the semimartingale differential of $H$ is
\begin{align*}
  \dd H_s
  &=-v(s)\,\dd Z_s-v'(s)Z_s\,\dd s+v(s)\,\dd a(s)\\
  &=-v(s)\sqrt{Z_s}\,\dd B_s-v'(s)Z_s\,\dd s.
\end{align*}
The continuous local martingale part of $H$ is
\[
  M_s^H=-\int_0^s v(r)\sqrt{Z_r}\,\dd B_r,
\]
and hence
\[
  \dd\langle M^H\rangle_s=v(s)^2Z_s\,\dd s.
\]
It\^o's formula for $N_s=e^{H_s}$ therefore gives
\[
  \dd N_s
  =N_s\,\dd H_s+\frac12N_s\,\dd\langle M^H\rangle_s
  =-v(s)N_s\sqrt{Z_s}\,\dd B_s
   +N_sZ_s\left(-v'(s)+\frac12v(s)^2\right)\dd s.
\]
Since $v'(s)=v(s)^2/2$, the drift vanishes.  Moreover,
\[
  0\leq N_s\leq e^{\lambda a(t)},
\]
so $N$ is a true martingale.  Since $N_0=1$ and
\[
  N_t
  =
  \exp\left\{
    -\lambda Z_t+\int_0^tu_{t-s}(\lambda)\,\dd a(s)
  \right\},
\]
taking expectations proves \eqref{e3.3}.

For the two-time formula, condition on $\cF_t$.  By the Markov and
branching properties, the mass present at time $t'$ can be decomposed,
conditionally on $\cF_t$, into two independent parts: a Feller
branching diffusion started from $Z_t$ at time $t$, with no further
immigration, and the mass generated by the immigration on $(t,t']$,
started from zero.  The first part has conditional Laplace transform
\[
  \exp\{-u_{t'-t}(\gamma)Z_t\},
\]
while applying \eqref{e3.3} to the shifted immigration path
$r\mapsto a(t+r)-a(t)$ gives for the second part
\[
  \exp\left\{-\int_t^{t'}u_{t'-s}(\gamma)\,\dd a(s)\right\}.
\]
Consequently,
\begin{align*}
  \E[e^{-\gamma Z_{t'}}\mid\cF_t]
  &=
  \exp\left\{
    -u_{t'-t}(\gamma)Z_t
    -\int_t^{t'}u_{t'-s}(\gamma)\,\dd a(s)
  \right\}.
\end{align*}
Multiplying by $e^{-\lambda Z_t}$ and taking expectations gives
\begin{align*}
  \E[e^{-\lambda Z_t-\gamma Z_{t'}}]
  &=
  \exp\left\{-\int_t^{t'}u_{t'-s}(\gamma)\,\dd a(s)\right\}
  \E\left[
    e^{-(\lambda+u_{t'-t}(\gamma))Z_t}
  \right].
\end{align*}
Applying \eqref{e3.3} to the last expectation with
$\lambda+u_{t'-t}(\gamma)$ gives \eqref{e3.4}.

Finally, $q_\eps(t)>0$.  Otherwise $Z_t\geq\eps$ a.s., so by \eqref{e3.3},
\[
  \int_0^tu_{t-s}(\lambda)\,\dd a(s)
  =-\log\E[e^{-\lambda Z_t}]
  \geq\lambda\eps.
\]
Hence
\begin{equation}\label{e7.1}
  \frac1\lambda\int_0^tu_{t-s}(\lambda)\,\dd a(s)
  \geq\eps.
\end{equation}
On the other hand, since $a$ is continuous, the Stieltjes measure $\dd a$ has no
atom at $t$, and dominated convergence gives
\[
 \frac1\lambda\int_0^tu_{t-s}(\lambda)\,\dd a(s)
  =\int_0^t\frac{2}{2+\lambda(t-s)}\,\dd a(s)
  \longrightarrow 0,
\]
as $\lambda\to\infty$, contradicting \eqref{e7.1}.
\end{proof}

\subsection{Lemma \ref{l3.4}}

\begin{proof}[Proof of Lemma \ref{l3.4}]
We prove the statement with an arbitrary level $\lambda>0$.
Set
\[
  F(t)=a(t)-\lambda t,\qquad
  m(t)=\min_{0\leq s\leq t}F(s).
\]
For $t\in(0,1]$,
\begin{align*}
  M_a(t)>\lambda
  &\Longleftrightarrow
    \exists\,h\in(0,t]:
    a(t)-a(t-h)>\lambda h\\
  &\Longleftrightarrow
    \exists\,s\in[0,t):
    F(t)>F(s)\\
  &\Longleftrightarrow
    F(t)>\min_{0\leq s\leq t}F(s)=m(t).
\end{align*}
  For the last equivalence,
continuity of $F$ ensures that its minimum on $[0,t]$ is attained;
if $F(t)>m(t)$, a minimizer cannot be $t$ and therefore lies in
$[0,t)$. 

\revision{Set
\[
  E_\lambda=\{t\in(0,1]:M_a(t)>\lambda\}.
\]
Since $F$ and $m$ are continuous,
$E_\lambda=\{t\in[0,1]:F(t)>m(t)\}$ is relatively
open in $[0,1]$ and does not contain $0$.
Let $(I_j)$ be its connected components, which are
at most countably many, and set
\[
  \alpha_j=\inf I_j,\qquad \beta_j=\sup I_j.
\]
Since $\alpha_j\notin E_\lambda$,
$F(\alpha_j)=m(\alpha_j)$.

For $t\in I_j$, a minimizer of $F$ on $[0,t]$
cannot lie in $(\alpha_j,t]$: such a point $s$
would satisfy $F(s)=m(s)$, contrary to
$s\in E_\lambda$. Thus $m(t)=m(\alpha_j)$,
and continuity shows that $m$ is constant on
$[\alpha_j,\beta_j]$. Consequently,
\[
  F(\beta_j)\geq m(\beta_j)
  =m(\alpha_j)=F(\alpha_j),
\]
and hence
\[
  \lambda|I_j|
  \leq a(\beta_j)-a(\alpha_j).
\]
The intervals $(\alpha_j,\beta_j]$ are pairwise
disjoint. Therefore,
\[
  \lambda|E_\lambda|
  \leq
  \sum_j\bigl(a(\beta_j)-a(\alpha_j)\bigr)
  \leq a(1)-a(0)=a(1).
\]}

Taking $\lambda=\delta$ proves \eqref{e3.9}.
\end{proof}

\bigskip

\qquad

\end{document}